\documentclass[12pt]{article}
\usepackage{graphicx}
\usepackage{amsmath}
\usepackage{amsfonts}
\usepackage{amsthm}
\usepackage{fancyhdr}
\usepackage{url}
\usepackage[margin=1in]{geometry}
\usepackage{color}

\numberwithin{equation}{section}

\begin{document}
\title{Towards characterization of all $3\times3$ extremal quasiconvex quadratic forms}
\author{Davit Harutyunyan and Graeme Walter Milton\\
\textit{Department of Mathematics, The University of Utah}}

\maketitle
\begin{abstract}
Given a $d\times d$ quasiconvex quadratic form, $d\geq 3,$ we prove that if the determinant of
its acoustic tensor is an irreducible extremal polynomial that is not identically zero, then the form itself is
an extremal quasiconvex quadratic form, i.e. it loses its quasiconvexity whenever a convex quadratic form is subtracted from it.
In the special case $d=3,$ we slightly weaken the condition, namely we prove, that if the determinant of the acoustic tensor of
the quadratic form is an extremal polynomial that is not a perfect square, then the form itself is an extremal quadratic form.
In the case $d=3$ we also prove, that if the determinant of the acoustic tensor of the form is identically zero, then the form
is either an extremal or polyconvex. Also, if the determinant of the acoustic tensor of the form is a perfect square, then the form
is either extremal, polycovex, or is a sum of a rank-one form and an extremal, whose acoustic tensor determinant is identically zero.
Here we use the notion of extremality introduced by Milton in [\ref{bib:Mil3}].\\

\newline
\linebreak
\end{abstract}

\newtheorem{Theorem}{Theorem}[section]
\newtheorem{Lemma}[Theorem]{Lemma}
\newtheorem{Corollary}[Theorem]{Corollary}
\newtheorem{Remark}[Theorem]{Remark}
\newtheorem{Definition}[Theorem]{Definition}
\newtheorem{Conjecture}[Theorem]{Conjecture}
\section{Introduction}
\label{sec:1}

Convexity is intimately tied with theromodynamic stability: if the free energy of a material is not a convex function of the state variables,
then when the material has state variables in the non-convex portion, it will naturally phase separate into two phases, each at the end of a tie
line bridging the non-convex portion, thus reducing the energy to a linear average of the energies of the two constituent phases (see
[\ref{bib:Cal}] and the introduction by Arthur Wightman in the book [\ref{bib:Isr}]).
For elasticity, among other examples, the picture is more complicated as when the material phase separates the displacement
field (ignoring the possiblity of cracking) must be continuous across the phase boundaries. Such phase separation is most easily
seen in shape memory materials such as Nitinol. The simplest of all geometries for the phase separated material is a laminate
of the two phases (also known as a twinned structure when the two phases are reflections of each other) and the continuity of the displacement field forces the difference of
the displacement gradient in one phase minus the displacement field in the second phase to be a rank-one tensor. Thus to avoid this layering transformation the
energy $f$ as a function of the displacement gradient  must be a rank one convex function i.e.
\begin{equation}
\label{0.1}
\theta f({A})+(1-\theta )f({B})\geq f(\theta{A}+(1-\theta){B})
\end{equation}
for all $\theta\in (0,1)$ and for all matrices ${A}$ and ${B}$ such that ${A}-{B}={x}\otimes{y}$ for some vectors ${x}$ and ${y}$. More generally to avoid
separation at the microscale into other geometries of possibly lower energy (with affine boundary conditions on the displacement $u$ at the boundary $\partial\Omega$ of the body, i.e.
${u}({x})={F}{x}$ for ${x}\in\partial\Omega$ for some fixed matrix ${F}$), the energy $f(\nabla{u}({x}))$ has to be a quasiconvex
function of $\nabla{u(x)}$. For more details about this theory of the stability of energy functions for non-linear elasticity the reader is referred to [\ref{bib:Ball1}].

The theory of quasiconvexity dates back even further to the pioneering work of Morrey [\ref{bib:Mor1}, \ref{bib:Mor2}] and has become central in the calculus of variations, material science, continuum mechanics, topology optimization, and in the theory of composites.
We refer to the book of Dacorogna [\ref{bib:Dac}] and the references therein for developments of the theory of quasiconvex analysis. In general it
is very difficult to determine if a non-quadratic function is quasiconvex, so Ball [\ref{bib:Ball1}] introduced the notion of polyconvexity, which is an intermediate condition between convexity and quasiconvexity: a function is polyconvex if it is a convex function of $\nabla u(x)$, its determinant, and other minors, these latter being the null-Lagrangians or quasi-affine functions, whose average
only depends on $F$. If $f(\xi)$ is quadratic then $f(\xi)$ is polyconvex if and only if it can be written as a sum of a convex function and a null-Lagrangian, which in the quadratic case is a linear combination of $2\times 2$ minors of the matrix $\xi\in\mathbb R^{N\times n}$,
[\ref{bib:Dac}, page 192, Lemma 5.27]. There exist quadratic forms that are quasiconvex but not polyconvex,
 as shown by Terpstra in [\ref{bib:Terp2}], which in the algebraic geometry language means, there exist nonnegative biquadratic forms
 that are not sums of squares, [\ref{bib:Pra}]. Explicit examples were given by Serre [\ref{bib:Ser1},\ref{bib:Ser2}], and later in [\ref{bib:Har.Mil}] we obtained an
especially simple example:
\begin{equation}
\label{1.1}
f(\xi)=\xi_{11}^2+\xi_{22}^2+\xi_{33}^2+\xi_{12}^2+\xi_{23}^2+\xi_{31}^2-2(\xi_{11}\xi_{22}+\xi_{22}\xi_{33}+\xi_{33}\xi_{11}).
\end{equation}

For linear elasticity  a necessary condition for a body
containing a linearly elastic homogeneous material with elasticity tensor $C$ to be stable
when the displacement is fixed at the boundary is the Legendre-Hadamard condition,
which is equivalent to rank one convexity for $C^2$ functions [\ref{bib:Dac}],
that is, the quadratic form associated with $C$, $f(\xi)=(C\xi;\xi)$
be rank-one convex, i.e.,
 $$ f(x\otimes y)=\sum_{ijk\ell=1}^3x_iy_jC_{ijk\ell}x_ky_\ell\geq 0\quad \forall x,y,$$
which is an algebraic condition. If one has equality for some non-zero $x,y$ then shear bands can form.
For general quadratic forms $f(\xi)=(M\xi;\xi)$, Van Hove [\ref{bib:VanHove1},\ref{bib:VanHove2}]
proved that rank-one convexity is equivalent to quasiconvexity. Due to this, and since we are only dealing with quadratic functions,
we will use the terms quasiconvexity and rank-one convexity interchangeably. Here we are interested in quadratic rank-one convex functions that are extremal in the sense of Milton ([\ref{bib:Mil3}, page 87, see also \ref{bib:Mil1}, section 25.2]):
\begin{Definition}
\label{def:1}
A quadratic quasiconvex form is called an extremal if one cannot subtract a rank-one form from it while preserving the quasiconvexity of the form. Here, a rank-one form is the square of a linear form.
\end{Definition}
Recall, that in traditional convex analysis an \textit{extreme} point of a convex set $S$ in a real vector space, is a point $Q\in S$ which does not lie in any open line segment joining two different points of $S.$ If one considers the convex cone $S$ of all $m\times n$ quasiconvex quadratic forms, then an \textit{extreme} point $Q$ of $S$ in the classical sense should be identified with the whole ray starting at the origin and passing through $Q.$ Of course, in the classical convex analysis setting, one can work with quadratic forms instead of rays by just defining an extreme point as a quasiconvex form that is not a sum of two linearly independent quasiconvex forms. As we are interested only in quasiconvex quadratic forms, in what follows we identify two quadratic forms that differ by a Null-Lagrangian as equivalent, i.e., we consider quadratic forms modulo Null-Lagrangians. The so-called trivial extreme points of $S$ are the rank-one forms, i.e., squares of linear forms. These are clearly not extremals in the sense of Milton, nevertheless, if a ray that is not a rank-one form is an extreme point of the convex cone $S,$ then it is apparently an extremal in the sense of Definition~\ref{def:1}. However, it is not known if the converse implication holds, i.e., if any extremal in the sense of Milton is a non-trivial extreme ray of $S.$ This seems to be a highly nontrivial task. Also, recall that the Krein-Milman theorem asserts the following: \textit{Let $X$ be a locally convex topological vector space, and let $K$ be a compact convex subset of $X.$ Then $K$ is the closed convex hull of its extreme points.} The notion of extremality in the sense of Milton has a strong Krein-Milman property, namely Milton [\ref{bib:Mil2}, \ref{bib:Mil4}] proved that any quasiconvex quadratic form is a sum of exactly one extremal and a polyconvex form. In the algebraic language, extremality reads as follows: \textit{A nonnegative biquadratic form $f(x\otimes y)=\sum_{ijk\ell=1}^3x_iy_jC_{ijk\ell}x_ky_\ell$ is extremal, if it can not be written as a perfect square plus a nonnegative form.} In the case of degree four nonnegative homogeneous polynomials that depend in $3$ variables, a related problem was considered and solved by Hilbert, [\ref{bib:Hil}], where he shows, that such polynomials are necessarily sums of perfect squares, thus there is no extremals among them. In the case of nonnegative homogeneous polynomials of even degree that depend on $n>2$ variables,
Hilbert proved, that they are not necessarily sums of squares of polynomials, but sums of squares of rational functions. The problem of
sums of squares is well-known in algebraic geometry and a recent significant result was obtained in [\ref{bib:Ble.Smi.Vel.}].
Milton presented an algorithm of finding such extremal forms in [\ref{bib:Mil2}]. Milton also proved new sharp inequalities by using extremal forms in [\ref{bib:Mil2}]. The first non-trivial explicit example of an extremal quasiconvex
function is that given in [\ref{bib:Har.Mil}] and is the function (\ref{1.1}) (which is also extremal in the two stronger senses of extremality introduced in that paper). Subsequently, in [\ref{bib:Har.Mil1}] we found a surprising and unexpected link between extremal quasiconvex functions with orthotropic symmetry and extremal polynomials (namely those polynomials that take only non-negative values, and which lose this property when any other non-negative polynomial
is subtracted from them). It is clear that for the quadratic form $f(\xi)=(M\xi;\xi),$ with $\xi\in\mathbb R^{N\times n},$ the expression for
$f(x\otimes y)$ can be written as $xC(y)x^T,$ where $x\in \mathbb R^N,$ $y\in\mathbb R^n$ and $C(y)$ is an
$N\times N$ symmetric matrix with its entries being quadratic forms in $y.$ The matrix $C(y)$ is called the acoustic
tensor of $f$ or just the $y-$matrix of $f(x\otimes y)$. The authors proved in [\ref{bib:Har.Mil1}],
that an elasticity quasiconvex tensor (a quasiconvex quadratic form $f(\xi)$ that depends solely on the symmetric part of $\xi$)
with orthotropic symmetry is extremal if the determinant of the $y-$matrix
(acoustic tensor) is an extremal polynomial that is not a perfect square, see [\ref{bib:Har.Mil1}, Theorem 5.1].

The goal of this work is to characterize all $3\times 3$ extremals, while our final goal is to characterize all $m\times n$ extremals.
We distinguish four main cases that exhaust all possibilities:
\begin{itemize}
\item[(i)] The determinant of the acoustic matrix of $f(x\otimes y)$ is an extremal polynomial that is not a perfect square.
\item[(ii)] The determinant of the acoustic matrix of $f(x\otimes y)$ is equivalently zero.
\item[(iii)] The determinant of the acoustic matrix of $f(x\otimes y)$ is an extremal polynomial that is a perfect square.
\item[(iv)] The determinant of the acoustic matrix of $f(x\otimes y)$ is not an extremal polynomial.
\end{itemize}
In this paper we extend the result proven in [\ref{bib:Har.Mil1}] to any $3\times 3$ quadratic form.
These of course include forms that need not be orthotropic and which can depend not just
on the symmetric part of $\xi,$ i.e., the strain in the elasticity context, but also on its antisymmetric part. In this
paper we study cases (i), (ii) and (iii) completely and conjecture on the case (iv). We believe, that case (iv) provides no extremals.
Some of the ideas of the proof already appeared in our work in [\ref{bib:Har.Mil1}].
As we aim to keep the present paper proof-self-contained, we shall repeat some of the ideas and proofs from [\ref{bib:Har.Mil1}].
It is rather striking, that our proof is a mixture of analytic and algebraic tools.
The problem under consideration is of algebraic geometry nature but comes from applications in composite
 materials and metamaterials and at first sight looks to be a part of quasiconvex analysis.
Some motivations for studying extremals were discussed in section 2 of [\ref{bib:Har.Mil1}]. Briefly, in the theory of composites one powerful method for obtaining bounds on effective tensor has proven to be the
translation method introduced by Tartar and Murat [\ref{bib:Tar0},\ref{bib:Mur.Tar},\ref{bib:Tar2}] and Lurie and Cherkaev [\ref{bib:LurCherk1}, \ref{bib:LurCherk2}]. These bounds are tightest if one
uses extremal quasiconvex forms as shown in [\ref{bib:Mil3}, page 87], see also [\ref{bib:Mil1}, section 25.2].
Allaire and Kohn [\ref{bib:All.Kohn}]  used certain type of extremals to bound the elastic energy of
two phase composites with isotropic phases. Kang and Milton [\ref{bib:Kan.Mil}] extended the translation method as a tool for bounding the volume fractions of materials in a two-phase body from boundary
measurements and in this context too it is natural to use extremal  quasiconvex forms (although it is an open question as to whether it is always best to use  extremal  quasiconvex forms in this context).
Extremal quasiconvex forms are also the best choice of quasiconvex functions for obtaining series expansions for effective tensors that have an extended domain of convergence, and thus analyticity properties as a function of the component moduli on this domain [see section 14.8, and page 373 of section 18.2 of \ref{bib:Mil1}]. Let us mention in conclusion, that we believe the study of extremals might lead to construction of new rank-one convex functions (unlike any in the literature) that are not quasiconvex, which is part of our future work.

\section{Extremal polynomials and determinants of quasiconvex quadratic forms: The main results}
\label{sec:3}

In this section we recall the notions of \textit{extremality and equivalence} of homogeneous polynomials, e.g., [\ref{bib:Pra}] and formulate
our main results. We start with a brief introduction.
\begin{Definition}
\label{def:3.1}
Assume $m$ and $n$ are natural numbers and $P(x_1,x_2,\dots,x_n)$ is a degree $2m$ homogeneous polynomial. Then $P(x)$ is called an extremal polynomial, if $P(x)\geq 0$ for all $x\in\mathbb R^n$ and $P(x)$ cannot be written as a sum of two other non-negative polynomials that are linearly independent.
\end{Definition}

\begin{Definition}
\label{def:3.2}
Assume $m$ and $n$ are natural numbers and $P(x_1,x_2,\dots,x_n)$ and $Q(x_1,x_2,\dots,x_n)$ are degree $2m$ homogeneous polynomials. Then $P(x)$ and $Q(x)$ are equivalent if there exists a non-singular matrix $A\in\mathbb R^{n\times n}$ such that $P(x)=Q(Ax).$
\end{Definition}

It is then straightforward to prove that this notion of equivalence is actually an equivalence relation preserving also the extremality of polynomials.

\begin{Theorem}
\label{th:3.3}
The notion of equivalence introduced in Definition~\ref{def:3.2} has the following properties:
\begin{itemize}
\item $P$ is equivalent to itself
\item If $P$ is equivalent to $Q$ then $Q$ is equivalent to $P$
\item If $P$ is equivalent to $Q$ and $Q$ is equivalent to $R,$ then $P$ is equivalent to $R$
\item If $P$ is equivalent to $Q$ and $Q$ is an extremal then $P$ is an extremal too
\end{itemize}
\end{Theorem}
We posed the question of characterizing all $3\times3 $ extremal quasiconvex quadratic forms in [\ref{bib:Har.Mil1}]
and provided a sufficient condition for elasticity tensors with orthotropic symmetry.
Our choice was motivated by some examples that are listed below:\\

\textbf{Example 1.} The form $f(\xi)=\xi_{11}^2+\xi_{22}^2+\xi_{33}^2$ has a determinant of its $y$-matrix equal to
$y_1^2y_2^2y_3^2$ which is an extremal polynomial, but $f(\xi)$ is obviously not an extremal.\\

\textbf{Example 2.} The form $f(\xi)=\xi_{11}^2+\xi_{22}^2$ has a determinant of its $y$-matrix equal to
$0,$ which is an extremal polynomial, but $f(\xi)$ is obviously not an extremal.\\

\textbf{Example 3.} The determinant of the $y$-matrix of any rank-one form is equivalently zero,
but a rank-one form is not an extremal.\\

\textbf{Example 4.} The most interesting and motivating example is the form in (\ref{1.1}) that appeared in [\ref{bib:Har.Mil}]:
 $$Q(\xi)=\xi_{11}^2+\xi_{22}^2+\xi_{33}^2-2(\xi_{11}\xi_{22}+\xi_{11}\xi_{33}+\xi_{22}\xi_{33})+\xi_{12}^2+\xi_{23}^2+\xi_{31}^2.$$
It turns out that the polynomial
$$P(y)=y_1^4y_2^2+y_2^4y_3^2+y_3^4y_1^2-3y_1^2y_2^2y_3^2$$
that is the determinant of the $y$-matrix of $Q(\xi),$ is an extremal polynomial that is not a perfect square, which is proven in
[\ref{bib:Har.Mil1}]. The form $Q(\xi)$ has been proven to be an extremal in even a
stronger sense than in Definition~\ref{def:1}, e.g., [\ref{bib:Har.Mil}, Theorem~1.5].
The following theorem is one of the main results of the paper.

\begin{Theorem}
\label{th:3.4}
Let $d\geq 3$ be a whole number and let the quadratic form $f(\xi)=\xi C\xi^T$ be quasiconvex,
where $\xi\in \mathbb R^{d\times d}$ and $C\in(\mathbb R^d)^4$ is a fourth order tensor.
Assume furthermore that the determinant of the $y-$matrix of $f(x,y)$ is an irreducible extremal polynomial.
Then $f$ is an extremal form.
\end{Theorem}

It turns out, that the irreducibility of the determinant can be replaced by a weaker condition in the case $d=3,$ namely the following
is true: \textit{Any reducible extremal polynomial of degree 6 in 3 variables is a perfect square.} However, we will not prove that statement
and will pursue a different way (that will also be useful for the subsequent observations) to prove the theorem in the case $d=3.$
\begin{Theorem}
\label{th:3.5}
Let the quadratic form $f(\xi)=\xi C\xi^T$ be quasiconvex,
where $\xi\in \mathbb R^{3\times 3}$ and $C\in(\mathbb R^3)^4$ is a fourth order tensor.
Assume furthermore that the determinant of the $y-$matrix of $f(x,y)$ is an extremal polynomial that is not a perfect square.
Then $f$ is an extremal form.
\end{Theorem}

The next theorem gives a complete characterization for case (ii).
\begin{Theorem}
\label{th:3.6}
Let the quadratic form $f(\xi)=\xi C\xi^T$ be quasiconvex,
where $\xi\in \mathbb R^{3\times 3}$ and $C\in(\mathbb R^3)^4$ is a fourth order tensor.
Assume furthermore that the determinant of the $y-$matrix of $f(x,y)$ is identically zero.
Then $f$ is either an extremal or a polyconvex form.
\end{Theorem}

The next theorem gives a complete characterization for case (iii).
\begin{Theorem}
\label{th:3.7}
Let the quadratic form $f(\xi)=\xi C\xi^T$ be quasiconvex,
where $\xi\in \mathbb R^{3\times 3}$ and $C\in(\mathbb R^3)^4$ is a fourth order tensor.
Assume furthermore that the determinant of the $y-$matrix of $f(x,y)$ is a perfect square.
Then $f$ is either extremal, polyconvex or a sum of a rank-one and extremal forms, where the
extremal has an identically zero determinant of the acoustic tensor.
\end{Theorem}
And finally, we conjecture, that case (iv) provides no extremals:
\begin{Conjecture}
\label{th:3.8}
Let the quadratic form $f(\xi)=\xi C\xi^T$ be quasiconvex,
where $\xi\in \mathbb R^{3\times 3}$ and $C\in(\mathbb R^3)^4$ is a fourth order tensor.
Assume furthermore that the determinant of the $y-$matrix of $f(x,y)$ is not an extremal polynomial.
Then $f$ is not an extremal either.
\end{Conjecture}
We emphasize, that this connection between extremal quasiconvex quadratic forms and extremal polynomials came out of the blue, after noticing certain connections. There may be a very deep algebraic geometric reason for this connection, but this eludes us at the present time. In conclusion we make the following observation: We will be considering $3\times 3$ quasiconvex quadratic forms. The case when the determinant of the acoustic tensor of the quadratic form is an extremal polynomial is clearly covered by Theorems~\ref{th:3.5}-\ref{th:3.7}. In the situation of Theorem~\ref{th:3.5} one gets an extremal quadratic form. In the situation of Theorem~\ref{th:3.6} there are two possibilities.
Let us mention, that it is straightforward to check whether a $3\times3$ quadratic form is polyconvex or not [\ref{bib:Dac}], thus Theorem~\ref{th:3.6} gives a complete algorithm to identity the whether the quadratic form is aa extremal (if it turns out to be non-polyconvex). In the situation of Theorem~\ref{th:3.7}, we can make the same observation and it only remains to understand whether the quadratic form is a sum of an extremal form with a zero acoustic tensor determinant and a rank-one form. It must be straightforward as well as Milton [\ref{bib:Mil3}] gives a formula for the rank-one form that can be subtracted from a quasiconvex quadratic form preserving the quasiconvexity. The situation when
the determinant is not an extremal polynomial is clearly supported by Conjecture~\ref{th:3.8}.

\section{Auxiliary lemmas}
\label{sec:4}
In this section we prove some auxiliary lemmas to be utilized in the proof of the main results.
Recall first a general fact from the theory of commutative algebras, e.g., [\ref{bib:Hun}, Theorem 6.14].
\begin{Theorem}
\label{th:4.1}
Every polynomial with real coefficients and in $n$ variables can be uniquely written as a product of irreducible polynomials
again with real coefficients.
\end{Theorem}
Next, let us mention, that in what follows $P(x)\ \vdots \ Q(x)$ or just $P\ \vdots \ Q$ means for polynomials $P(x)$ and $Q(x),$ that
$Q(x)$ divides $P(x).$ Utilizing the above theorem, let us prove the following simple lemma.
\begin{Lemma}
\label{lem:4.2}
Assume $x\in\mathbb R^d$ and $A(x)=(a_{ij}(x)),$ $i=1,\dots, m,\  j=1,\dots, n$ is an $m\times n$ matrix with polynomial coefficients, such, that each term
$a_{ij}(x)$ is a homogeneous polynomial of degree $2p,$ where $m,n,d,p\in\mathbb N.$ If $\mathrm{rank}(A(x))\leq 1$ for all $x\in\mathbb R^d$, then there exist homogeneous polynomials
$b_i(x)$ and $c_i(x),$ such that $a_{ij}(x)=b_i(x)c_j(x),$ for $i=1,\dots, m, \ j=1,\dots, n.$
\end{Lemma}
\begin{proof}
We prove the statement by induction doing the induction for $m+n.$ For $m+n=2$ there is nothing to prove.
Assume the statement is true for $m+n=L$ and let us prove it for $m+n=L+1.$ We can without loss of generality assume, that $n\geq 2$
as otherwise we can transpose the matrix $A(x).$ It is clear that the left $m\times (n-1)$ block of $A(x)$
has a rank at most $1,$ thus by induction $A(x)$ has the form
\begin{equation}
\label{4.1}A(x)=
\begin{bmatrix}
b_1(x)c_1(x) & b_1(x)c_2(x) &\dots & b_1(x)c_{n-1}(x) & a_{1n}(x)\\
b_2(x)c_1(x) & b_2(x)c_2(x) &\dots & b_2(x)c_{n-1}(x) & a_{2n}(x)\\
\cdot & \cdot & \dots \cdot & \cdot & \cdot\\
\cdot & \cdot & \dots \cdot & \cdot & \cdot\\
\cdot & \cdot & \dots \cdot & \cdot & \cdot\\
b_m(x)c_1(x) & b_m(x)c_2(x) &\dots & b_m(x)c_{n-1}(x) & a_{mn}(x)\\
\end{bmatrix}.
\end{equation}
First of all observe, that if $A(x)$ has a zero row or column, then by removing it we reduce the sum $m+n$ and thus finish the proof by induction
by taking the factor $b_i(x)$ or $c_j(x)$ of that row or column to be identically zero. Thus we assume, that all polynomials $c_j(x)$
are nonzero. Next, it is clear, that one can without loss of generality assume,
that the greatest common factor of the polynomials $b_1(x),b_2(x),\dots,b_m(x)$ is $1,$
as otherwise, it can be absorbed in each of $c_i(x),$ for $i=1,2,\dots,{n-1}.$
We have from the property $\mathrm{rank}(A(x))\leq 1,$ that
$$b_1(x)c_j(x)a_{in}(x)=b_i(x)c_j(x)a_{1n}(x),\quad\text{for all}\quad i=1,2,\dots,m,$$
thus as $c_j(x)\neq 0,$ we obtain
\begin{equation}
\label{4.2}
b_1(x)a_{in}(x)=b_i(x)a_{1n}(x),\quad\text{for all}\quad i=1,2,\dots,m.
\end{equation}
Assume $p(x)$ is any irreducible factor of $b_1(x)$ that appears with the power $\alpha$ in $b_1(x),$ i.e., $b_1(x)=p^\alpha(x)q(x).$
We then get from (\ref{4.2}), that
\begin{equation}
\label{4.3}
p^\alpha(x)q(x)a_{in}(x)=b_i(x)a_{1n}(x),\quad\text{for all}\quad i=1,2,\dots,m.
\end{equation}
Since, $\mathrm{GCF}(b_1(x),b_2(x),\dots,b_m(x))=1,$ then there exists an index $i\in \{1,2,\dots,d\}$
such that $b_i(x)$ is not divisible by $p(x),$ thus we get from (\ref{4.3}), that $a_{1n}(x)\ \vdots \ p^\alpha(x)$.
As $p(x)$ was any factor, we get $a_{1n}(x)\ \vdots \ b_1(x),$ i.e,
\begin{equation}
\label{4.4}
a_{1n}(x)=b_1(x)c_{n}(x).
\end{equation}
From the equality (\ref{4.2}) we get either $b_1(x)\equiv 0$ or $a_{in}(x)=b_i(x)c_{n}(x),$ for $i=2,3,\dots,m.$ In the first case
the matrix $A(x)$ has a row of zeros, thus by induction we are done, in the second case we get the desired representation
 of $A(x).$
\end{proof}
\begin{Lemma}
\label{lem:4.3}
Assume $A=A(x),$ where $x=(x_1,x_2,x_3)$ is a $3\times3$ symmetric matrix the entries of which are quadratic forms in $x.$ Assume furthermore that
\begin{equation*}
\mathrm{det}(A(x))\equiv 0\quad\text{for all}\quad x\in\mathbb R^3,
\end{equation*}
and that the cofactor matrix $A_{cof}$ of $A$ has nonnegative elements on the diagonal.
Then $A_{cof}$ has one of the forms shown below:
\begin{equation}
\label{4.7}
\begin{bmatrix}
P & \alpha P & \beta P\\
\alpha P & \alpha^2 P & \alpha\beta P\\
\beta P & \alpha\beta P & \beta^2P
\end{bmatrix},\quad
 \begin{bmatrix}
P_1^2 & P_1Q_1 & P_1R_1\\
P_1Q_1 & Q_1^2 & Q_1R_1\\
P_1R_1 & Q_1R_1 & R_1^2
\end{bmatrix},\qquad
\begin{bmatrix}
l_1^2S & l_1l_2S & l_1l_3S\\
l_1l_2S & l_2^2S & l_2l_3S\\
l_1l_3S & l_2l_3S & l_3^2S
\end{bmatrix},\quad \alpha,\beta\in\mathbb R.
\end{equation}
\end{Lemma}

\begin{proof}
The identity
$$\mathrm{det}(A(x))\equiv 0\quad\text{for all}\quad x\in\mathbb R^3$$
implies that
\begin{equation}
\label{4.8}
\mathrm{rank}(A_{cof}(x))\leq 1\quad\text{for all}\quad x\in\mathbb R^3.
\end{equation}
The proof is now based on the following fact:\\
\textit{Assume  $P(x)$ $Q(x)$ and $R(x)$ are homogeneous polynomials of degree four such that
$P(x),Q(x)\geq 0$ and $P(x)Q(x)=R^2(x)$ for all $x\in\mathbb R^3.$ Then one of the three situations is true:
\begin{itemize}
\item[(i)] $P(x)\equiv Q(x)$ and $R(x)=\pm P(x),$
\item[(i)] $P(x)=P_1^2(x),$ $Q(x)=Q_1^2(x)$ and $R(x)=P_1(x)Q_1(x),$
\item[(i)] $P(x)=S(x)l_1^2(x),$ $Q(x)=S(x)l_2^2(x)$ and $R(x)=l_1(x)l_2(x)S(x).$
\end{itemize}}
The above fact is a direct consequence of the unique representation of polynomials as a product of irreducible multipliers.
To finish the proof, we consider all 3 principal $2\times 2$ minors of the matrix $A(x),$ for each of which the product of the diagonal
terms is the square of the off-diagonal term, thus the above fact can be utilized to obtain the structure of the matrix $A(x)$.
\end{proof}
The next lemma is a sign-changing property of indefinite quadratic forms.
\begin{Lemma}
\label{lem:4.4}
Assume $f(\xi)$ is an indefinite quadratic form in $n$ variables that vanishes at a point $(\xi_1^0,\xi_2^0,\dots,\xi_n^0).$ Then given any
open neighbourhood $U$ of the point $\xi^0=(\xi_1^0,\xi_2^0,\dots,\xi_n^0),$ there exist two open subsets $U_1,U_2\subset U$ of $U,$ such that
$$f(\xi)<0,\quad \xi\in U_1\quad\text{and}\quad f(\xi)>0,\quad \xi\in U_2.$$
\end{Lemma}

\begin{proof}
We can without loss of generality assume, that $f(\xi)$ has canonical form and thus due to its indefiniteness, $f$ has the form
\begin{equation}
\label{4.9}
f(\xi)=\sum_{i=1}^m \xi_i^2-\sum_{i=m+1}^{m+k} \xi_i^2,
\end{equation}
where $1\leq m<m+k\leq n.$ Assume $\epsilon$ and $\delta$ are small numbers.
We perturb the points $\xi_i^0$ to $\xi_i^0+\epsilon,$ for $i=0,1,\dots,m$ and to $\xi_i^0+\delta,$ for $i=m+1,\dots,m+k$
and denote
$$\xi_{\epsilon,\delta}=(\xi_1^0+\epsilon,\dots,\xi_m^0+\epsilon,\xi_{m+1}^0+\delta,\dots,\xi_{m+k}^0+\delta,\xi_{m+k+1}^0+\epsilon,\dots, \xi_{n}^0+\epsilon).$$
It is clear, that if $\epsilon$ and $\delta$ are small enough, them $\xi_{\epsilon,\delta}\in U.$ We have at the new point
$$f(\xi_{\epsilon,\delta})=2\epsilon(\xi_1^0+\cdots\xi_m^0)-2\delta(\xi_{m+1}^0+\cdots\xi_{m+k}^0)+m\epsilon^2-k\delta^2.$$
Denote $S_1=\xi_1^0+\cdots\xi_m^0$ and $S_2=\xi_{m+1}^0+\cdots\xi_{m+k}^0$. If $S_1\neq 0,$ then we can choose $\delta=\epsilon^2$ and choosing the sign of
$\epsilon$ we can make the values of $f$ both, positive and negative in some open subsets of $U.$ The case $S_2\neq 0$ is analogous. If now $S_1=S_2=0,$
then we have $f(\xi_{\epsilon,\delta})=m\epsilon^2-k\delta^2,$ which can again be made both positive and negative in open subsets of $U$ for small
$\epsilon$ and $\delta$ by choosing either $\delta=\epsilon^2$ or $\epsilon=\delta^2.$
\end{proof}

The last lemma will be crucial in the proof of Theorem~\ref{th:3.6}.
\begin{Lemma}
\label{lem:4.5}
Let $P(x)$ be a quadratic form and let $Q(x)$ be a fourth order nonnegative homogeneous polynomial in $x\in\mathbb R^3.$ Assume $\delta>0$ and that
for any $t\in[0,\delta]$ there exists a quadratic form $R_t(x)$ such, that $P^2(x)-tQ(x)=R_t^2(x).$ Then $Q(x)=\alpha^2P^2(x)$ for some $\alpha\in\mathbb R.$
\end{Lemma}
\begin{proof}
Consider two cases.\\
\textbf{Case 1: The form $P(x)$ is indefinite.} Fix any $t\in (0,\delta).$ It is clear, that if $P(x)=0,$ then the condition $tQ(x)+R_t^2(x)=0$ and the nonnegativity of $Q(x)$ imply, that $R_t(x)=0,$ thus Marecellini's theorem (see Theorem~\ref{th:6.3}, Section~\ref{sec:6}) implies, that
$R_t(x)=\lambda P(x)$ for all $x\in\mathbb R^3$ and some $\lambda\in\mathbb R,$ which then gives $Q(x)=\alpha^2P^2(x).$\\
\textbf{Case 2: The form $P(x)$ is definite.} In this case we can without loss of generality assume, that $P(x)$ is positive semi-definite and it
has a canonical form. If $P(x)=x_1^2,$ then the inequality $x_1^4-\delta Q(x)\geq 0$ for all $x_1\in\mathbb R$ implies $Q(x)=\alpha x_1^4,$ thus we
are done. If $P(x)=x_1^2+x_2^2,$ then first of all the inequality $(x_1^2+x_2^2)^2-\delta Q(x)\geq 0$ for all $x_1,x_2\in\mathbb R$ implies, that $Q(x)$ depends only on $x_1$ and $x_2,$ i.e., $Q(x)=ax_1^4+2bx_1^3x_2+cx_1^2x_2^2+2dx_1x_2^3+ex_2^4.$ We then have
\begin{align}
\label{4.10}
P^2(x)-tQ(x)&=(1-ta)x_1^4-2tbx_1^3x_2+(2-tc)x_1^2x_2^2-2tdx_1x_2^3+(1-te)x_2^4\\ \nonumber
&=R_t^2(x_1,x_2)\\ \nonumber
&=(\sqrt{1-ta}x_1^2-\frac{tb}{\sqrt{1-ta}}x_1x_2+\sqrt{1-te}x_2^2)^2,
\end{align}
from which we get $d=\frac{b\sqrt{1-te}}{\sqrt{1-ta}}.$ We aim to show that $b=d=0.$ If $b=0,$ then the equality $d=\frac{b\sqrt{1-te}}{\sqrt{1-ta}}$ implies $d=0.$ If $b\neq 0,$ then we get $\frac{d}{b}=\frac{\sqrt{1-te}}{\sqrt{1-ta}},$ which means, that the ratio $\frac{\sqrt{1-te}}{\sqrt{1-ta}}$ is constant, which is possible only if $a=e$ and thus we get $b=d$ as well. Equating now the coefficients of $x_1^2x_2^2$ of both sides of (\ref{4.10}), we get
$$(2-tc)(1-ta)=b^2t^2+2(1-ta)^2,\quad t\in(0,\delta),$$
thus we obtain
$$
t(2a-c)+t^2(ac-2a^2-b^2)=0,\quad t\in(0,\delta),
$$
which then implies $c=2a$ and $b=0.$ Thus we get in all cases, that $b=d=0,$ therefore (\ref{4.10}) reduces to
\begin{equation}
\label{4.11}
P^2(x)-tQ(x)=(1-ta)x_1^4+(2-tc)x_1^2x_2^2+(1-te)x_2^4=(\sqrt{1-ta}x_1^2+\sqrt{1-te}x_2^2)^2,
\end{equation}
which gives from the equality of the coefficients of $x_1^2x_2^2,$
$$4t(a+e-c)+t^2(c^2-4ae)=0,\quad t\in(0,\delta),$$
and thus $c=a+e$ and $c^2=4ae,$ which then gives $(a-e)^2=0,$ thus $a=e$ and $c=2a,$ and we finally get $Q(x)=aP^2(x).$
In the last case when $P(x)=x_1^2+x_2^2+x_3^2,$ we can substitute $x_3=0$ in the equality $P^2(x)-tQ(x)=R_t^2(x)$ to get $P^2(x_1,x_2,0)-tQ(x_1,x_2,0)=R_t^2(x_1,x_2,0)$, from where we already know, that the $(x_1,x_2)$ part of $Q$ has the form $a(x_1^2+x_2^2)^2.$
Similarly, the $(x_1,x_3)$ and $(x_2,x_3)$ parts of $Q$ have the forms $a(x_1^2+x_3^2)^2$ and $a(x_2^2+x_3^2)^2,$ which then implies, that $Q$ has the
desired form $aP^2(x).$ The proof is finished now.
\end{proof}

\section{Extremal quadratic forms of arbitrary dimension}
\label{sec:5}
In this section we prove the main result of the paper.

\begin{proof}[Proof of Theorem~\ref{th:3.4}]
Assume in contradiction that $f(x,y)$ is not an extremal, then there exists a nonzero rank-one form $(xBy^T)^2$ such that
$$f(x,y)-(xBy^T)^2\geq 0\qquad\text{for all}\qquad x,y\in\mathbb R^d.$$
From the inequality
$$f(x,y)-(xBy^T)^2\geq 0$$
 we get that
 \begin{equation}
 \label{5.1}
 f(x,y)-t(xBy^T)^2\geq 0\quad\text{ for all }\quad t\in [0,1]\quad\text{and}\quad x,y\in\mathbb R^d.
 \end{equation}
Denote now by $T(y)$ the $y-$matrix of the biquadratic form $f(x,y)$ and by $T_t(y)$ the $y-$matrix of the biquadratic form $f(x,y)-t(xBy^T)^2.$
Inequality (\ref{5.1}) now implies that the $y-$matrix of the form $f(x,y)-t(xBy^T)^2,$ i.e., the matrix $T_t(y)$ is positive semi-definite for all $y\in\mathbb R^d$ and $t\in[0,1].$
We recall the Brunn-Minkowski inequality for determinants [\ref{bib:Bec.Bel},\ref{bib:Bel}], which we apply in the next step.
\begin{Theorem}[Brunn-Minkowski inequality]
\label{th:Minkowsky}
Assume $n\in\mathbb N$ and $A$ and $B$ are $n\times n$ symmetric positive semi-definite matrices. Then the following inequality holds:
$$(\mathrm{det}(A+B))^{1/n}\geq(\mathrm{det}(A))^{1/n}+(\mathrm{det}(B))^{1/n}.$$
\end{Theorem}

The equality
$$T(y)=[T(y)-T_t(y)]+[T_t(y)],$$
the positive semi-definiteness of the matrices $T(y)-T_t(y)$ and $T_t(y),$ and the Brunn-Minkowski inequality give the following estimate:
$$\left(\mathrm{det}(T(y))\right)^{1/d}\geq \left(\mathrm{det}(T(y)-T_t(y))\right)^{1/d}+\left(\mathrm{det}(T_t(y))\right)^{1/d},$$
or
\begin{equation}
 \label{5.2}
\mathrm{det}(T(y))\geq\mathrm{det}(T_t(y)).
\end{equation}
By direct calculation we obtain
\begin{equation}
\label{5.3}
\mathrm{det}(T_t(y))=\mathrm{det}(T(y))-t\sum_{i,j=1}^ds_{i}s_{j}(T_{cof}(y))_{ij},
\end{equation}
where $s_i=\sum_{j=1}^db_{ij}y_j.$
Inequality (\ref{5.2}) together with the relation (\ref{5.3}) imply
$$\mathrm{det}(T(y))\geq\mathrm{det}(T(y))-t\sum_{i,j=1}^ds_{i}s_{j}(T_{cof}(y))_{ij}\geq 0,\qquad t\in[0,1].$$
As by the requirement of the theorem $\mathrm{det}(T(y))$ is not identically zero and for $t=0$ the right hand side of the last inequality is
exactly $\mathrm{det}(T(y))$, then by the extremality of $\mathrm{det}(T(y))$ the right hand side must be a multiple of $\mathrm{det}(T(y)),$ i.e.,
$$\mathrm{det}(T(y))-t\sum_{i,j=1}^ds_{i}s_{j}(T_{cof}(y))_{ij}=\lambda(t)\mathrm{det}(T(y)),$$
which gives
\begin{equation}
\label{5.4}
\sum_{i,j=1}^ds_{i}s_{j}(T_{cof}(y))_{ij}=\frac{1-\lambda(t)}{t}\mathrm{det}(T(y)),\qquad\text{for}\qquad t\in(0,1].
\end{equation}
Both parts of the equality (\ref{5.4}) are polynomials in $y=(y_1,y_2,\dots,y_d)$ thus the expression $\frac{1-\lambda(t)}{t}$ must be constant, therefore
(\ref{5.4}) reduces to
\begin{equation}
\label{5.5}
\sum_{i,j=1}^ds_{i}s_{j}(T_{cof}(y))_{ij}=\alpha\cdot \mathrm{det}(T(y)),
\end{equation}
for some constant $\alpha\in\mathbb R.$ The positive semi-definiteness of $T(y)$ implies positive semi-definiteness of the cofactor matrix $T_{cof}(y),$ thus
$$\sum_{i,j=1}^ds_{i}s_{j}(T_{cof}(y))_{ij}\geq 0\qquad\text{for all}\qquad y\in\mathbb R^d.$$
We have on the other hand $\mathrm{det}(T(y))\geq0,$ and $\mathrm{det}(T(y))$ is not identically zero, thus $\alpha\geq 0.$
Consider now two different cases:\\

\textbf{Case 1: $\alpha=0.$} In this case identity (\ref{5.5}) becomes
$$
\sum_{i,j=1}^ds_{i}s_{j}(T_{cof}(y))_{ij}\equiv 0.
$$
Recalling the positive semi-definiteness of $T_{cof}(y)$ we get the system of equalities:
\begin{equation}
\label{5.6}
\sum_{j=1}^ds_j(T_{cof}(y))_{ij}\equiv 0,\qquad i=1,2,\dots,d.
\end{equation}
Consider now the linear system
$$s_j=0,\quad j=1,2,\dots,d.$$
As the matrix $B$ has a rank at least $1,$ the solution to the above linear system is a
proper subspace $V$ of $\mathbb R^d,$ i.e., is a set with measure zero in $\mathbb R^d.$
Therefore the columns of the cofactor matrix $T_{cof}(y)$ are linearly dependent a.e. in $\mathbb R^d$ i.e.,
 $\mathrm{det}(T_{cof}(y))=0$ a.e. in $\mathbb R^d,$ and thus by continuity we get
$$\mathrm{det}(T_{cof}(y))\equiv 0.$$
 Taking the determinant of the identity $T(y)T_{cof}(y)^T=\mathrm{det}(T(y))I$ we get  $\mathrm{det}(T(y))\equiv 0,$  which is a contradiction.
\textbf{Case 1} is now proved.\\

\textbf{Case 2: $\alpha>0.$}  In this case identity (\ref{5.5}) implies
\begin{equation}
\label{5.7}
 \mathrm{det}(T(y))=k\sum_{i,j=1}^ds_{i}s_{j}(T_{cof}(y))_{ij},\qquad k>0.
\end{equation}
Our goal is now to prove that $(\ref{5.7})$ implies that $\mathrm{det}(T(y))$ is a reducible polynomial.
Consider the biquadratic form $g(x,y)=f(x,y)-k(xBy^T)^2.$ Denote by $G(y)$ the $y$-matrix of $g.$ By formula (\ref{5.3}) we have
$$\mathrm{det}(G(y))=\mathrm{det}(T(y))-k\sum_{i,j=1}^ds_{i}s_{j}(T_{cof}(y))_{ij},$$
thus owing to (\ref{5.7}) we get
\begin{equation}
\label{5.8}
\mathrm{det}(G(y))\equiv 0.
\end{equation}
Utilizing again formula (\ref{5.3}) we obtain from the identity $f(x,y)=g(x,y)+k(xBy^T)^2,$
$$
\mathrm{det}(T(y))=\mathrm{det}(G(y))+k\sum_{i,j=1}^ds_{i}s_{j}(G_{cof}(y))_{ij},
$$
thus taking into account (\ref{5.8}) we get
\begin{equation}
\label{5.9}
\mathrm{det}(T(y))=k\sum_{i,j=1}^ds_{i}s_{j}(G_{cof}(y))_{ij}.
\end{equation}
Next we apply Lemma~\ref{lem:4.2} to the matrix $G.$ The condition $\mathrm{det}(G(y))\equiv 0$ in (\ref{5.8})
implies, that
\begin{equation}
\label{5.10}
\mathrm{rank}(G_{cof}(y))\leq 1,\quad\text{for all}\quad y\in\mathbb R^d,
\end{equation}
thus we get from Lemma~\ref{4.2}, that
$(G_{cof}(y))_{ij}=a_i(y)b_j(y)$ for some polynomials $a_i(y)$ and $b_j(y),$ $i,j=1,2,\dots,d.$ We then obtain from (\ref{5.9}),
that
\begin{equation}
\label{5.11}
\mathrm{det}(T(y))=k\left(\sum_{i=1}^da_i(y)s_i(y)\right)\left(\sum_{j=1}^db_j(y)s_j(y)\right),
\end{equation}
which contradicts the irreducibility of the polynomial $\mathrm{det}(T(y))$. The theorem is proven now.
\end{proof}

\section{Quadratic forms with dimension $d=3.$}
\label{sec:6}
In this section we prove Theorems~\ref{th:3.5}-\ref{th:3.7}.
\begin{proof}[Proof of Theorem~\ref{th:3.5}] We aim to apply Lemma~\ref{lem:4.3} to the matrix $G_{cof}(y)$, for which
we must verify that the diagonal terms of the cofactor matrix $G_{cof}(y)$ are nonnegative.
To prove that $(G_{cof}(y))_{33}\geq 0$ for all $y\in\mathbb R^3,$ we can without
loss of generality assume that $(G_{cof}(y))_{33}\neq0$ for a.e. $y\in\mathbb R^3,$ which means that the first
two rows of $G(y)$ are linearly independent for a.e. $y\in\mathbb R^3.$ As $\mathrm{det}(G(y))\equiv 0,$
then the last row of $G(y)$ is a linear combination of the first two, i.e., the matrix $G(y)$ has the form:
\begin{equation}
\label{6.1} G(y)=
\begin{bmatrix}
g_{11} & g_{12} & ag_{11}+bg_{12} \\
g_{12} & g_{22} & ag_{12}+bg_{22} \\
ag_{11}+bg_{12} & ag_{12}+bg_{22} & a^2g_{11}+b^2g_{22}+2abg_{12}
\end{bmatrix},
\end{equation}
where
\begin{equation}
\label{6.2}
a(y)=-\frac{(G_{cof}(y))_{13}}{(G_{cof}(y))_{33}}\qquad\text{and}\qquad b(y)=-\frac{(G_{cof}(y))_{23}}{(G_{cof}(y))_{33}}
\end{equation}
are rational functions defined a.e. in $\mathbb R^3.$ From (\ref{6.1}) and (\ref{6.2}) we get
\begin{equation}
\label{6.3} G_{cof}(y)=
\begin{bmatrix}
a^2(G_{cof})_{33} & ab(G_{cof})_{33} & -a(G_{cof})_{33} \\
ab(G_{cof})_{33} & b^2(G_{cof})_{33} & -b(G_{cof})_{33} \\
-a(G_{cof})_{33} & -b(G_{cof})_{33} & (G_{cof})_{33}
\end{bmatrix}.
\end{equation}
Plugging in the values of the entries of the cofactor matrix $G_{cof}(y)$ from formula (\ref{6.3}) into (\ref{5.9}) and utilizing (\ref{6.2}) we arrive at
\begin{equation}
\label{6.4}
\mathrm{det}(T(y))=k\frac{(s_1(G_{cof}(y))_{13}-s_2(G_{cof}(y))_{23}-s_3(G_{cof}(y))_{33})^2}{(G_{cof})_{33}}.
\end{equation}
Therefore from the inequality $\mathrm{det}(T(y))\geq 0$ and the fact that $\mathrm{det}(T(y))$ is not identically zero we get the desired inequality
$(G_{cof}(y))_{33}\geq 0.$ The nonnegativity of the other diagonal terms of the cofactor matrix is analogous.
The requirements of Lemma~\ref{lem:4.3} are now satisfied, thus the matrix $G_{cof}(y)$ has one
of the forms shown in (\ref{4.7}). For the first form in (\ref{4.7}) we get by (\ref{5.9})
$$\mathrm{det}(T(y))=k(G_{cof}(y))_{11}(s_1+\alpha s_2+\beta s_3)^2.$$
The multiplier $(G_{cof}(y))_{11}$ is a homogeneous nonnegative polynomial of degree four in three variables,
thus by Hilbert's theorem [\ref{bib:Hil},\ref{bib:Pra}] it can be expressed as a sum of perfect squares, which
implies that $\mathrm{det}(T(y))$ is either a perfect square or not an extremal. For the second form
in (\ref{4.7}) we have
$$\mathrm{det}(T(y))=k(s_1P_1+s_2Q_1+s_2Q_3)^2,$$
thus $\mathrm{det}(T(y))$ is a perfect square. For the third form in (\ref{4.7}) we get
$$\mathrm{det}(T(y))=kS(s_1l_1+s_2l_2+s_3l_3)^2.$$
By the same argument as for the first case we get that $S(y)\geq 0$ for all $y\in\mathbb R^3,$ thus being a convex quadratic form
it can be written as the sum of squares of linear forms, which yields that $\mathrm{det}(T(y))$ is either a perfect square or not an
extremal polynomial. The proof is finished.
\end{proof}

\begin{proof}[Proof of Theorem~\ref{th:3.6}]
Before starting the proof, let us recall the following classical result of Terprstra, which will
be utilized in the sequel.
\begin{Theorem}[Terpstra]
\label{th:6.1}
Any $2\times n$ quasiconvex quadratic form is necessarily polyconvex, where $n\in\mathbb N.$
\end{Theorem}
Let us emphasize that a quadratic form is polyconvex if and only if it is a sum of a convex form and a Null-Lagramgian (which
is a linear combination of all $2\times 2$ minors of the matrix $\xi$), e.g., [\ref{bib:Dac}]. Thus, a quadratic form is polyconvex,
if and only if, the biquadratic form obtained from it by substituting a rank-one matrix $x\otimes y$ in place of $\xi$
is a sum of squares of linear combinations of $x_iy_j.$ Hence, the lemma below is a corollary of
Terprstra's result and the above observation, which will be a key factor in the proof of Theorem~\ref{th:3.6}.
\begin{Lemma}
\label{lem:6.2}
Assume $A(y)$ $B(y)$ and $C(y)$ are quadratic forms in $n$ variables, such that $A(y)$ and $C(y)$ are positive semidefinite
and $A(y)C(y)\geq B^2(y)$ for all $y\in\mathbb R^n.$ Then the $2\times n$ form $f(x,y)=x_1^2A(y)+2x_1x_2B(y)+x_2^2C(y)$
is convex in the variables $\xi_{ij}=x_iy_j.$
\end{Lemma}

\begin{proof}
By Theorem~\ref{th:6.1}, it is sufficient to check, that the form $f(x,y)$ is quasiconvex, which is trivial, as its $y-$matrix
has nonnegative diagonal terms $A(y)$ and $C(y)$ and a nonnegative determinant $A(y)C(y)-B^2(y)\geq 0.$
\end{proof}
Consider now 2 main cases.\\
\textbf{Case 1: The cofactor matrix $T_{cof}(y)$ has a zero diagonal element.}\\
First of all observe, that if one of the diagonal elements of $T(y)$ is identically zero, then
$f(x,y)$ automatically becomes a $2\times 3$ form and thus by Terprstra's result above, its quasiconvexity implies
that it is polyconvex. Assume now $T(y)$ has second degree polynomial elements on the main diagonal and
that $(T_{cof})_{33}(y)\equiv 0.$ By the positive semi-definiteness of $T_{cof}(y),$ we then have
that $(T_{cof})_{13}(y)=(T_{cof})_{23}(y)\equiv 0.$ The last equalities imply, that the matrix
obtained from $T(y)$ by removing the last row, has a rank at most one, thus by Lemma~\ref{lem:4.2}
the matrix $T(y)$ has either of the forms:
\begin{equation}
\label{6.5}
\begin{bmatrix}
P & \alpha P & \beta P\\
\alpha P & \alpha^2 P & \alpha\beta P\\
\beta P & \alpha\beta P & Q
\end{bmatrix},\quad
 \begin{bmatrix}
P_1^2 & P_1Q_1 & P_1R_1\\
P_1Q_1 & Q_1^2 & Q_1R_1\\
P_1R_1 & Q_1R_1 & S
\end{bmatrix}\quad \alpha,\beta\in\mathbb R,
\end{equation}
where $\mathrm{deg}(P)=\mathrm{deg}(Q)=\mathrm{deg}(S)=2$ and $\mathrm{deg}(P_1)=\mathrm{deg}(Q_1)=\mathrm{deg}(R_1)=1$ and $\alpha\neq 0.$ In the first case
we have, that $P(y)$ is a positive semi-definite quadratic form and
$$0\leq (T_{cof})_{11}(y)=\alpha^2P(y)(Q(y)-\beta^2P(y)),$$
thus
$$Q(y)-\beta^2P(y)\geq 0,$$
which implies, that the form $Q(y)-\beta^2P(y)$ is convex, i.e., it is a sum of squares, therefore
$$Q(y)=\beta^2P(y)+\sum_{i=1}^3a_i^2(y).$$
We then obtain, that
$$f(x,y)=P(y)(x_1+\alpha x_2+\beta x_3)^2+\sum_{i=1}^3a_i^2(y)x_3^2,$$
and $P(y)$ is a convex quadratic form, therefore $f(\xi)$ is polyconvex. In the second case we get similarly, that
$$ S(y)=R_1^2(y)+\sum_{i=1}^3a_i^2(y),$$
therefore
$$f(x,y)=(P_1(y)x_1+Q_1(y)x_2+R_1(y)x_3)^2+\sum_{i=1}^3a_i^2(y)x_3^2,$$
and thus is polyconvex.\\
\textbf{Case 2: All diagonal elements of the cofactor matrix $T_{cof}(y)$ are nonzero.}\\
In that case we have denoting $T(y)=(t_{ij}(y))_{i,j=1}^3$ that, since $\mathrm{det}T(y)\equiv 0,$
then the last row of $T(y)$ is a linear combination of the first two, which means that
$T(y)$ has the form of $G(y)$ shown in (\ref{6.1}), hence we get
\begin{equation}
\label{6.6} T(y)=
\begin{bmatrix}
t_{11} & t_{12} & at_{11}+bt_{12} \\
t_{12} & t_{22} & at_{12}+bt_{22} \\
at_{11}+bt_{12} & at_{12}+bt_{22} & a^2t_{11}+b^2t_{22}+2abt_{12}
\end{bmatrix},
\end{equation}
where
\begin{equation}
\label{6.7}
a(y)=-\frac{(T_{cof}(y))_{13}}{(T_{cof}(y))_{33}}\qquad\text{and}\qquad b(y)=-\frac{(T_{cof}(y))_{23}}{(T_{cof}(y))_{33}}.
\end{equation}
On the other hand the matrix $T(y)$ satisfies the requirements of Lemma~\ref{lem:4.3}, thus
the cofactor matrix $T_{cof}(y)$ has one of the forms
\begin{equation}
\label{6.8}
\begin{bmatrix}
P & \alpha P & \beta P\\
\alpha P & \alpha^2 P & \alpha\beta P\\
\beta P & \alpha\beta P & \beta^2P
\end{bmatrix},\quad
 \begin{bmatrix}
P_1^2 & P_1Q_1 & P_1R_1\\
P_1Q_1 & Q_1^2 & Q_1R_1\\
P_1R_1 & Q_1R_1 & R_1^2
\end{bmatrix},\qquad
\begin{bmatrix}
l_1^2S & l_1l_2S & l_1l_3S\\
l_1l_2S & l_2^2S & l_2l_3S\\
l_1l_3S & l_2l_3S & l_3^2S
\end{bmatrix},\quad \alpha,\beta\in\mathbb R.
\end{equation}
Consider now each case separately.\\
\textbf{Case 2a: The cofactor matrix $T_{cof}(y)$ has the first form shown in (\ref{6.8}).}\\
In the first case we have $a(y)=-\frac{1}{\beta}=a\in\mathbb R$ and $b(y)=-\frac{\alpha}{\beta}=b\in\mathbb R,$ thus
we get
\begin{equation}
\label{6.9}
f(x,y)=t_{11}(y)(x_1+ax_3)^2+2t_{12}(y)(x_1+ax_3)(x_2+bx_3)+t_{22}(y)(x_2+bx_3)^2.
\end{equation}
It is clear, that $t_{11},t_{22}\geq 0$ and $t_{11}t_{22}\geq t_{12}^2$ which is a consequence of the quasiconvexity of the form
$f(x,y)$ for $x_3=0,$ thus denoting $X_1=x_1+ax_3$ and $X_2=x_2+bx_3,$ we get by Lemma~\ref{lem:6.2}, that the form
\begin{equation*}
f(x,y)=t_{11}(y)X_1^2+2t_{12}(y)X_1X_2+t_{22}(y)X_2^2
\end{equation*}
is convex in the variables $X_iy_j$, and thus the form $f(\xi)$ is polyconvex in the variables $\xi_{ij}=x_iy_j,$ which completes the proof for the first case.\\
\textbf{Case 2b: The cofactor matrix $T_{cof}(y)$ has the third form shown in (\ref{6.8}).}\\
In this case we have
\begin{equation}
\label{6.10}
f(x,y)=\frac{1}{l_3^2}\left[t_{11}(x_1l_3-x_3l_1)^2+2t_{12}(x_1l_3-x_3l_1)(x_2l_3-x_3l_2)+t_{22}(x_2l_3-x_3l_2)^2\right],
\end{equation}
where $l_1,l_2,l_3\neq 0.$ Let us write the conditions that ensure, that $f(x,y)$ is a
polynomial in $x$ and $y.$ We have, that the coefficients of $x_1x_3$ and $x_2x_3$ and
are quadratic forms in $y,$ which yields the following conditions:
\begin{equation}
\label{6.11}
(t_{11}l_1+t_{12}l_2)\ \vdots\ l_3,\qquad (t_{12}l_1+t_{22}l_2)\ \vdots\ l_3.
\end{equation}
If now the linear forms $l_1$ and $l_2$ are divisible by $l_3,$ i.e., $l_1, l_2 \ \vdots\ l_3$, then we have $l_1=\alpha l_3$ and $l_2=\beta l_3,$ thus
the cofactor matrix will have the first form shown in (\ref{6.8}), which yields the polyconvexity of $f.$ If $l_1$ is divisible by $l_3$ and
$l_2$ is not, then we get from (\ref{6.11}), that $t_{12}\ \vdots\ l_3$ and $t_{22}\ \vdots\ l_3,$ thus from the positive semi-definiteness of $t_{22}$ we get
\begin{equation}
\label{6.12}
t_{22}=\alpha^2l_3^2,\qquad t_{12}=ll_{3},
\end{equation}
where $l=l(y)$ is a linear form in $l_3.$ It is clear, that $\alpha>0,$ thus we obtain from the condition $t_{11}t_{22}-t_{12}^2\geq 0$, that
\begin{equation}
\label{6.13}
t_{11}\geq \frac{l^2}{\alpha^2}.
\end{equation}
The condition (\ref{6.13}) implies the the quadratic form $t_{11}-\frac{l^2}{\alpha^2}$ is convex, thus it can be written as a sum of squares, i.e.,
\begin{equation}
\label{6.14}
t_{11}=\frac{l^2}{\alpha^2}+\sum_{i=1}^3a_i^2.
\end{equation}
Putting now $l_1=al_3$ we get from (\ref{6.10}),
\begin{equation}
\label{6.15}
f(x,y)=\left[\frac{l(x_1-ax_3)}{\alpha}+\alpha(l_3x_2-l_2x_3)\right]^2+\sum_{i=1}^3[a_i(x_1-ax_3)]^2,
\end{equation}
which implies the polyconvexity of $f.$ The last and the most tricky subcase here is when both $l_1$ and $l_2$ are not divisible by $l_3.$
Assume, that the form $f$ is not an extremal. We aim then to prove that it must be polyconvex. Assume the form
$G(x,y)=f(x,y)-\left(\sum_{i=1}^3x_is_i(y)\right)^2$ is quasiconvex. Then from the proof of Thereom~\ref{th:3.5} we have, that
$\mathrm{det}(T(y))=S(l_1s_1+l_2s_2+l_3s_3)^2,$ where $T(y)$ is the acoustic matrix of $f(x,y).$ Therefore, we obtain the condition
\begin{equation}
\label{6.16}
l_1s_1+l_2s_2+l_3s_3\equiv 0.
\end{equation}
We again recall Treprstra's result (Lemma~\ref{lem:6.2}) for the quadratic forms $A(y)=t_{11}(y),$ $C(y)=t_{22}(y)$ and $B(y)=t_{12}(y),$ which gives
that the form $t_{11}(y)z_1^2-2t_{12}(y)z_1z_2+t_{22}(y)z_2^2$ must be convex in the variables $\xi_{ij}=z_iy_j,$ thus it is a sum of squares, i.e.,
$$t_{11}(y)z_1^2-2t_{12}(y)z_1z_2+t_{22}(y)z_2^2=\sum_{i=1}^3(a_i(y)z_1-b_i(y)z_2)^2,$$
where $a_i(y)$ and $b_i(y)$ are linear forms in $y.$ Therefore, the last equality implies
\begin{equation}
\label{6.17}
t_{11}(y)=\sum_{i=1}^3a_i^2(y),\quad t_{22}(y)=\sum_{i=1}^3b_i^2(y),\quad t_{12}(y)=\sum_{i=1}^3a_i(y)b_i(y).
\end{equation}
Next we obtain some divisibility conditions. First we have, that since the form $G(x,y)=f(x,y)-\left(\sum_{i=1}^3x_is_i(y)\right)^2$ is quasiconvex thus its $33$ cofactor element must be nonnegative, i.e.,
$$(t_{11}-s_1^2)(t_{22}-s_2^2)-(t_{12}-s_1s_2)^2\geq 0,$$
which gives
$$t_{11}s_2^2+t_{22}s_1^2-2t_{12}s_1s_2\leq t_{11}t_{22}-t_{12}^2,$$
which can be rewritten as follows:
\begin{equation}
\label{6.18}
\sum_{i=1}^3(a_is_2-b_is_1)^2\leq l_3^2S.
\end{equation}
The last inequality implies, that if $l_3=0,$ then $a_is_2-b_is_1=0$ for $i=1,2,3,$ thus since $l_3$ is a linear form, we discover
\begin{equation}
\label{6.19}
(a_is_2-b_is_1)\ \vdots \ l_3,\quad i=1,2,3.
\end{equation}
Next, we have from (\ref{6.10}) and from (\ref{6.17}) the following representation of $f(x,y):$
\begin{equation}
\label{6.20}
f(x,y)=\sum_{i=1}^3\left[a_ix_1+b_ix_2-\frac{1}{l_3}(a_il_1+b_il_2)x_3\right]^2,
\end{equation}
where we aim to show, that $(a_il_1+b_il_2)\ \vdots \ l_3$ for $i=1,2,3,$ which will evidently ensure the polyconvexity of $f.$
From (\ref{6.19}) we have $(l_2a_is_2-l_2b_is_1)\ \vdots \ l_3$ and from (\ref{6.16}) we have $(l_2s_2+l_1s_1)\ \vdots \ l_3,$ thus we get from the last two
divisibility conditions, that $s_1(a_il_1+b_il_2)\ \vdots \ l_3$ for $i=1,2,3,$ hence, if $s_1$ is not divisible by $l_3,$ then we are done.
If now $s_1\ \vdots \ l_3,$ then we get from (\ref{6.16}), that $s_2\ \vdots \ l_3$ too. Similarly, as none of the forms $l_1$ and $l_2$ is zero,
we get the following set of divisibility conditions:
\begin{equation}
\label{6.21}
s_i\ \vdots \ l_j,\quad\text{if}\quad i\neq j,
\end{equation}
Note, that in the case when two of $l_i$ and $l_j$ are linearly dependent, we are done as shown in (\ref{6.11})-(\ref{6.15}).
The last observation then contradicts the conditions (\ref{6.21}) unless $s_1=s_2=s_3\equiv 0,$ which itself is a contradiction as we are assuming
the subtracted form $\left(\sum_{i=1}^3x_is_i(y)\right)^2$ is not zero. The proof of \textbf{Case 2b} is finished now.

\textbf{Case 2c: The cofactor matrix $T_{cof}(y)$ has the second form shown in (\ref{6.8}).}\\
We again prove, that if $f(\xi)$ is not an extremal, then it must be polyconvex. Assume $f(\xi)$ is not an extremal, thus
the biquadratic form $g(x,y)=f(x,y)-(x_1s_1(y)+x_2s_2(y)+x_3s_3(y))^2$ is nonnegative, where at least one of the linear forms
$s_1(y), s_1(y)$ and $s_3(y)$ is nonzero. Invoking the idea in the proof of Theorem~\ref{th:3.4}, we consider the bilinear form
$$f_h(x,y)=f(x,y)-h(x_1s_1(y)+x_2s_2(y)+x_3s_3(y))^2,$$
for all $h\in[0,1].$ It is then clear, that $f_h(\xi)$ is quasiconvex for all $h\in[0,1].$ On the other hand we have from (\ref{5.2}), that
$$0=\mathrm{det}(T(y))\geq \mathrm{det}(T_h(y))\geq 0,$$ thus we get
\begin{equation}
\label{6.23}
\mathrm{det}(T_h(y))=0,\quad\text{for all}\quad h\in[0,1],
\end{equation}
therefore the cofactor matrix $\mathrm{cof}(T_h(y))$ must have one of the forms shown in (\ref{6.8})
for all $h\in[0,1]$. If for a value $h_0\in(0,1)$ the matrix $\mathrm{cof}(T_{h_0}(y))$ has the first form shown in (\ref{6.8}), then by
\textbf{Case 2a} the form $f_{h_0}(x,y)$ is polyconvex and so is the form $f(x,y)=f_{h_0}(x,y)+h_0(x_1s_1(y)+x_2s_2(y)+x_3s_3(y))^2$.
If for a value $h_0\in(0,1)$ the matrix $\mathrm{cof}(T_{h_0}(y))$ has the third form shown in (\ref{6.8}), then by
\textbf{Case 2b} the form $f_{h_0}(x,y)$ is polyconvex as it is not an extremal, and so is the form $f(x,y)=f_{h_0}(x,y)+h_0(x_1s_1(y)+x_2s_2(y)+x_3s_3(y))^2$. We assume then, that the matrix $\mathrm{cof}(T_{h}(y))$ has the second form shown in (\ref{6.8}) for all $h\in[0,1],$ i.e.,
\begin{equation}
\label{6.24}\mathrm{cof}(T_{h}(y))=
\begin{bmatrix}
P_h^2 & P_hQ_h & P_hR_h\\
P_hQ_h & Q_h^2 & Q_hR_h\\
P_hR_h & Q_hR_h & R_h^2
\end{bmatrix},\quad\text{for all}\quad h\in[0,1),
\end{equation}
where we have moreover,
\begin{align}
\label{6.25}
P_h^2&=P^2-h(s_2^2t_{33}+s_3^2t_{22}-2s_2s_3t_{23}),\\ \nonumber
Q_h^2&=Q^2-h(s_1^2t_{33}+s_3^2t_{11}-2s_1s_3t_{13}),\\ \nonumber
R_h^2&=R^2-h(s_1^2t_{22}+s_2^2t_{11}-2s_1s_2t_{12}).
\end{align}
Applying now Lemma~\ref{lem:4.5} to each of the identities in (\ref{6.25}), we get, that
\begin{align}
\label{6.26}
s_3^2t_{22}+s_2^2t_{33}-2s_2s_3t_{23}&=\alpha^2P^2=\alpha^2(t_{22}t_{33}-t_{23}^2),\\ \nonumber
s_1^2t_{33}+s_3^2t_{11}-2s_1s_3t_{13}&=\beta^2Q^2=\beta^2(t_{11}t_{33}-t_{13}^2),\\ \nonumber
s_1^2t_{22}+s_2^2t_{11}-2s_1s_2t_{12}&=\gamma^2R^2=\gamma^2(t_{11}t_{22}-t_{12}^2).
\end{align}
Next, we consider the following subcases:\\
\textbf{Case 2ca: $\gamma=0$ and $s_1\neq 0.$} In this case we have, that if $s_2\equiv 0,$ then $s_1^2t_{22}\equiv 0,$
thus $t_{22}\equiv 0,$ thus by the positivity $(T_{cof})_{33}(y)=t_{11}t_{22}-t_{12}^2=-t_{12}^2\geq 0,$ we obtain
 $(T_{cof})_{33}(y)\equiv 0$, which puts us in \textbf{Case 1,} thus $f(\xi)$ is polyconvex. If now $s_2\neq 0,$ then we get
from the last equality in (\ref{6.26}), that
$$s_1^2t_{22}+s_2^2t_{11}-2s_1s_2t_{12}\equiv 0,$$
which gives $2s_1s_2t_{12}=s_1^2t_{22}+s_2^2t_{11}.$ By the geometric and arithmetic mean inequality we have
$$4s_1^2s_2^2t_{12}^2=(s_1^2t_{22}+s_2^2t_{11})^2\geq 4s_1^2s_2^2t_{11}t_{22},$$ from where we obtain
$t_{11}t_{22}-t_{12}^2\leq 0,$ and thus by positivity $(T_{cof})_{33}(y)=t_{11}t_{22}-t_{12}^2\equiv 0$, which again
brings us beck to the situation in \textbf{Case 1,} and hence $f(\xi)$ is polyconvex. Next, observe, that
since the form $x_1s_1(y)+x_2s_2(y)+x_3s_3(y)$ is not identically zero, then we can assume, that $s_1\neq 0,$ therefore the cases $\gamma=0$ or $\beta=0$ are done. We can then assume, that $\beta,\gamma\neq 0.$ The next case is the following:\\
\textbf{Case 2cb: $\alpha=0.$} We have already understood, that if in this case $s_2\neq 0$ or $s_3\neq 0,$ then we are done. Assume now, that
$s_2=s_3\equiv 0.$ The formula (\ref{5.3}) and the identity (\ref{6.23}) imply
\begin{equation}
\label{6.27}
\sum_{i,j=1}^3s_i(y)s_j(y)(T_{cof})_{ij}(y)=\mathrm{det}(T_1(y))=0,
\end{equation}
which gives $s_1^2(y)(T_{cof})_{11}(y)\equiv 0,$ hence $(T_{cof})_{11}(y)\equiv 0,$ which is again \textbf{Case 1,} and hence $f(\xi)$ is polyconvex.
We can now assume, that we are in the following situation:\\
\textbf{Case 2cc: $\alpha\beta\gamma\neq 0$ and $s_1\neq 0.$} If $s_2=s_3\equiv 0,$ then we get from (\ref{6.27}) and the following observation, that
$f(\xi)$ is polyconvex. Assume, now $s_2\neq 0.$ A straightforward manipulation transfers the equalities in (\ref{6.26}) to the following form:
\begin{align}
\label{6.28}
\left(t_{11}-\frac{s_1^2}{\gamma^2}\right)\left(t_{22}-\frac{s_2^2}{\gamma^2}\right)=\left(t_{12}-\frac{s_1s_2}{\gamma^2}\right)^2,\\ \nonumber
\left(t_{11}-\frac{s_1^2}{\beta^2}\right)\left(t_{33}-\frac{s_3^2}{\beta^2}\right)=\left(t_{13}-\frac{s_1s_3}{\beta^2}\right)^2,\\ \nonumber
\left(t_{22}-\frac{s_2^2}{\alpha^2}\right)\left(t_{33}-\frac{s_2^2}{\alpha^2}\right)=\left(t_{23}-\frac{s_2s_3}{\alpha^2}\right)^2.
\end{align}
Observe, that from the first equality of (\ref{6.28}) one must have one of the cases:
$$t_{11}-\frac{s_1^2}{\gamma^2}=S,\quad \quad t_{22}-\frac{s_2^2}{\gamma^2}=\lambda^2S,\quad t_{12}-\frac{s_1s_2}{\gamma^2}=\lambda S,$$
where $S$ is a quadratic form and $\lambda\in\mathbb R,$ or
$$t_{11}-\frac{s_1^2}{\gamma^2}=\delta u^2,\quad t_{22}-\frac{s_2^2}{\gamma^2}=\delta v^2,\quad t_{12}-\frac{s_1s_2}{\gamma^2}=\delta uv,$$
where $\delta^2=1.$
In the first case we have by direct calculation, that
$$t_{11}t_{22}-t_{12}^2=S\frac{(s_2-\lambda s_1)^2}{\gamma^2}=R^2,$$
thus one must have $S=s^2,$ for some linear form $s.$ Therefore, we have in all cases,
$$
t_{11}=\frac{s_1^2}{\gamma^2}+\delta u^2,\quad t_{22}=\frac{s_2^2}{\gamma^2}+\delta v^2,\quad t_{12}=\frac{s_1s_2}{\gamma^2}+\delta uv,
$$
where $\delta=\pm 1.$ We have on the other hand, that
$$t_{11}t_{22}-t_{12}^2=\frac{\delta}{\gamma^2}(s_1v-s_2u)^2\geq 0,$$
from where, and from the fact that the cofactor element $t_{11}t_{22}-t_{12}^2$ is not identically zero, we get $\delta=1.$ Concluding, we get
the following representation of $t_{11}, t_{22}$ and $t_{12}:$
\begin{equation}
\label{6.29}
t_{11}=\frac{s_1^2}{\gamma^2}+u^2,\quad t_{22}=\frac{s_2^2}{\gamma^2}+v^2,\quad t_{22}=\frac{s_1s_2}{\gamma^2}+uv.
\end{equation}
It is also clear, that all 6 multipliers on the left hand side of equation (\ref{6.28}) satisfy similar identities. Denote
$\lambda=\frac{1}{\max(\alpha, \beta, \gamma)}$ and consider the form
$$f_\lambda(x,y)=f(x,y)-\lambda(x_1s_1(y)+x_2s_2(y)+x_3s_3(y))^2.$$ We aim to prove, that the biquadratic form $f_\lambda(x,y)$ is nonnegative.
To that end, note first of all, that the equation (\ref{6.27}) and the formula (\ref{5.3}) for the determinant of $f_\lambda$ imply, that
\begin{equation}
\label{6.30}
\mathrm{det}(T_\lambda(y))=0.
\end{equation}
Second, note, that (\ref{6.25}) and (\ref{6.26}) imply, that
$$P_h^2=P^2(1-h\alpha)\geq0, \quad Q_h^2=Q^2(1-h\beta)\geq0, \quad Q_h^2=Q^2(1-h\gamma)\geq0,\quad\text{for all}\quad h\in[0,1],$$
from where we get
\begin{equation}
\label{6.31}
0<\alpha,\beta,\gamma\leq 1,\quad\text{and}\quad \lambda\geq 1.
\end{equation}
Next, the diagonal elements of the matrix $T_\lambda(y)$ are $t_{11}-\lambda s_1^2,$ $t_{22}-\lambda s_2^2$ and  $t_{33}-\lambda s_3^2,$ respectively,
which are nonnegative due to the conditions (\ref{6.29}), (\ref{6.31}) and the definition of $\lambda.$
Again, by (\ref{6.25}) and (\ref{6.26}), the principal minors of $T_\lambda(y)$ are $P^2(1-\lambda\alpha)$, $Q^2(1-\lambda\beta)$ and
$=Q^2(1-\lambda\gamma)$, which are obviously nonnegative. Concluding, we obtain that the form $f_\lambda$ is quasiconvex. Moreover one of the
diagonal elements of the cofactor matrix $\mathrm{cof}(T_\lambda(y))$ is identically zero, which again goes back to the \textbf{Case 1} for $f_\lambda$ implying, that $f_\lambda$ is polyconvex and therefore so is the form $f(x,y)=f_\lambda(x,y)+\lambda(x_1s_1(y)+x_2s_2(y)+x_3s_3(y))^2.$
The proof of the theorem is finished now.

\end{proof}

\begin{proof}[Proof of Theorem~\ref{th:3.7}]. To begin with, let us assume that $f(\xi)$ is not an extremal.
Recall that then the analysis in the beginning of Section~\ref{sec:5} goes through, thus we have the following identities
\begin{equation}
\label{6.32}
f(x,y)=g(x,y)+(x_1s_1(y)+x_2s_2(y)+x_3ks_3(y))^2,
\end{equation}
and
\begin{equation}
\label{6.33}
\mathrm{det}(G(y))\equiv 0,
\end{equation}
where $G(y)$ is the $y-$matrix of the form $g(x,y)$ and $s_i(y),$ $i=1,2,3,$ are linear
forms in $y.$ We have seen furthermore in (\ref{6.1})-(\ref{6.5}), that
the diagonal terms of the cofactor matrix $G_{cof}(y)$ are nonnegative. Consider now three different cases:\\
\textbf{Case 1. One of the diagonal elements of $G(y)$ is positive semi-definite.}\\
 In this case we have by the Sylvester's criterion, that the matrix $G(y)$ is positive semi-definite, thus the form $g(x,y)$ is quasiconvex with an
 identically zero determinant of its acoustic matrix, therefore Theorem~\ref{th:3.6} immediately implies, that the form $g(\xi)$ is a sum of a
 rank-one form and an extremal that has an identically zero determinant.\\
\textbf{Case 2. One of the diagonal elements of $G(y)$ is indefinite.}\\
We can without loss of generality assume, that the element $g_{11}(y)$ is indefinite. Let us show, that $g_{11}(y)=0$ implies $g_{22}(y)=0$ and
$g_{33}(y)=0.$ Indeed, assume by contradiction, that $g_{22}(y^0)>0$ and $g_{11}(y^0)=0,$ for some $y^0\in\mathbb R^3.$ Then
by continuity, $g_{22}(y)>0$ in a neighbourhood $U$ of the point $y^0$ and by Lemma~\ref{lem:4.4}, we have $g_{11}(y^1)<0,$ for some point $y^1\in U,$
thus $g_{11}(y^1)g_{22}(y^1)<0,$ which contradicts the inequality $(G_{cof}(y^1))_{33}=g_{11}(y^1)g_{22}(y^1)-g_{12}^2(y^1)\geq 0.$ The proof
of the case $g_{11}(y^1)>0$ is analogous, as Lemma~\ref{lem:4.4} provides both positive and negative values of $g_{11}.$ Next we apply the following result of
Marcellini, [\ref{bib:Marcellini}, Corollary 1]: the original formulation is below.
\begin{Theorem}[Marcellini]
\label{th:6.3}
Let $f$ and $g$ be two quadratic forms in $\mathbb R^n,$ with $g$ indefinite. If $f(\xi)=0$ for every $\xi$ such that $g(\xi)=0,$
then there exists $\lambda\in\mathbb R$ such that $f=\lambda g.$
\end{Theorem}
Now, due to the above theorem and the nonnegativity of the diagonal elements of the cofactor matrix, we get that the exist constants $\lambda,\mu\geq0,$
such, that
\begin{equation}
\label{6.34}
g_{22}(y)=\lambda^2g_{11}(y),\qquad g_{33}(y)=\mu^2g_{11}(y).
\end{equation}
Note, that the case $\lambda=0$ or $\mu=0$ reduces to the first case, thus we can assume, that $\lambda,\mu>0.$
The inequality
$$(G_{cof}(y))_{33}=g_{11}(y)g_{22}(y)-g_{12}^2(y)\geq 0,$$
implies, that if $g_{11}(y)=0$ then $g_{12}(y)=0$, thus we again get from the indefiniteness of $g_{11}(y)$, that $g_{12}(y)=\alpha g_{11}(y).$ Similarly,
we get $g_{13}(y)=\beta g_{11}(y)$ and $g_{23}(y)=\gamma g_{11}(y),$ thus the matrix $G(y)$ has the form
\begin{equation}
\label{6.35}
G(y)=g_{11}(y)
\begin{bmatrix}
1 & \alpha & \beta \\
\alpha & \lambda^2 & \gamma\\
\beta & \gamma & \mu^2
\end{bmatrix}=g_{11}(y)\cdot A.
\end{equation}
Again, the condition
$$0\leq (G_{cof}(y))_{33}=g_{11}(y)g_{22}(y)-g_{12}^2(y)=g_{11}^2(y)(\lambda^2-\alpha^2),$$
and the indefiniteness of $g_{11}(y)$ imply
\begin{equation}
\label{6.36}
\lambda^2\geq \alpha^2.
\end{equation}
The condition
$$0=\mathrm{det}(G(y))=g_{11}^3(y)\mathrm{det}(A)$$
and the indefiniteness of $g_{11}(y)$ imply
\begin{equation}
\label{6.37}
\mathrm{det}(A)=0.
\end{equation}
Combining conditions (\ref{6.36}) and (\ref{6.37}), we get, that the matrix $A$ is positive semi-definite by Sylvester's criterion,
and we have the following representation of $f,$
\begin{equation}
\label{6.38}
f(x,y)=g_{11}(y)xAx^T+(x_1s_1(y)+x_2s_2(y)+x_3s_3(y))^2.
\end{equation}
If one of the linear forms $s_i(y)$, say $s_2(y)$ is identically zero, then $f(x,y)$ will become negative for some $y$ that makes $g_{11}(y)$ negative
and for $x=(0,x_2,0)$ with big enough values of $x_2,$ thus the linear forms $s_i(y)$ are nonzero. We aim the to show, that the bilinear form
$x_1s_1(y)+x_2s_2(y)+x_3s_3(y)$ separates in the $x$ and $y$ variables, i.e., each $s_i(y)$ is a scalar multiple of $s_1(y).$
Indeed, first of all it is clear that the solution of the linear equation $s_1(y)=0$
is a proper subspace of $R^3,$ thus has a zero measure. Therefore, by Lemma~\ref{lem:4.4}, we can choose $y^0$ such that $g_{11}(y^0)<0$ and $s_1(y^0)\neq 0.$
By continuity, there exists $\epsilon>0,$ such that
\begin{align}
\label{6.39}
g_{11}(y^0)&<0\quad\text{if}\quad y\in B_{\epsilon}(y^0)\\ \nonumber
s_1(y^0)&\neq 0\quad\text{if}\quad y\in B_{\epsilon}(y^0),
\end{align}
where $B_{\epsilon}(y^0)\in\mathbb R^3$ is the ball centered at $y^0$ and with radius $\epsilon.$ Next, we choose
$$x_1=-x_2\frac{s_2(y)}{l_s(y)}-x_3\frac{s_3(y)}{s_1(y)},$$
that makes the bilinear form $x_1s_1(y)+x_2s_2(y)+x_3s_3(y)$ vanish. Due to the quasiconvexity of $f,$ the condition $g_{11}(y)<0$ and the positive semi-definiteness of the matrix $A,$ we must have $xAx^T=0$ as long as $x_1=-x_2\frac{s_2(y)}{s_1(y)}-x_3\frac{s_3(y)}{s_1(y)}.$ We aim now to prove,
that the latter is possible only, if $s_2$ and $s_3$ are constant multiples of $s_1.$ The solution of $xAx^T=0$ is a proper subspace of $\mathbb R^3,$
thus if we assume, that $\frac{s_2(y)}{s_1(y)}$ is not constant, then the solution of $\frac{s_2(y)}{s_1(y)}=c$ must be a proper subspace of $R^3$ for any
$c\in\mathbb R,$ thus if $y$ runs over the ball $B_{\epsilon}(y^0),$ the ratio $\frac{s_2(y)}{s_1(y)}$ takes infinitely many values as the ball has a positive measure and therefore the solution $x$ of $x_1=-x_2\frac{s_2(y)}{s_1(y)}-x_3\frac{s_3(y)}{s_1(y)}$ when $y$ runs over $B_{\epsilon}(y^0),$
can not be contained in a proper subspace of $R^3.$ Therefore we have $s_2(y)=as_1(y)$ and $s_3(y)=bs_1(y),$ and hence we arrive at the formula
\begin{equation}
\label{6.40}
f(x,y)=g_{11}(y)xAx^T+s_1^2(y)(x_1+ax_2+bx_3)^2.
\end{equation}
It is now clear, that if $x_1+ax_2+bx_3=0$, then $xAx^T=0$, thus since $x_1+ax_2+bx_3$ is linear, is must divide the quadratic form $xAx^T$, therefore
from the positivity of $xAx^T$ we get $xAx^T=c(x_1+ax_2+bx_3)^2,$ which then yields
\begin{equation}
\label{6.41}
f(x,y)=(s_1^2(y)+cg_{11}(y))(x_1+ax_2+bx_3)^2.
\end{equation}
Finally, the positivity of $f(x,y)$ implies the positivity of the quadratic form $s_1^2(y)+cg_{11}(y)$ and thus the polyconvexity of $f(\xi).$\\
\textbf{Case 3. All diagonal elements of $G(y)$ are negative semi-definite.}\\
Denote $H(y)=-G(y),$ thus $H(y)$ is positive semi-definite and
\begin{equation}
\label{6.42}
f(x,y)=(x_1s_1(y)+x_2s_2(y)+x_3s_3(y))^2-xH(y)x^T.
\end{equation}
Observe, that if $s_1(y)=0$ for some $y\in\mathbb R^3,$ then choosing $x=(x_1,0,0)$ we have
$$f(x,y)=-h_{11}(y)x_1^2\geq 0,$$
thus $h_{11}(y)=0$, which implies, that $h_{11}(y)$ is divisible by $s_1(y)$ and by the positivity of $h_{11}(y),$ we get
$h_{11}(y)=a_{11}s_1^2(y).$ Similarly we have $h_{22}(y)=a_{22}s_2^2(y)$ and $h_{33}(y)=a_{33}s_3^2(y),$ where $a_{11},a_{22},a_{33}\geq 0.$
Next, we have from the inequality $h_{ii}(y)h_{jj}(y)\leq h_{ij}^2(y)$ thus $h_{ij}(y)=a_{ij}s_i(y)s_j(y)$ and thus the form $xH(y)x^T$
depends only on the variables $z_1=x_1s_1(y), z_2=x_2s_2(y), z_3=x_3s_3(y),$ which gives
\begin{equation}
\label{6.43}
f(x,y)=zAz^T,
\end{equation}
for some matrix $A.$ If $s_1\equiv 0,$ then the form $xH(y)x^T$ does not depend on $x_1$ and thus
we end up with a $2\times 3$ quasiconvex form, therefore by Theorem~\ref{th:6.1} it will be polyconvex. If none of the linear forms
$s_i(y)$ is identically zero, then we can choose $y\in\mathbb R^3$ such that $s_i(y)\neq 0,$ and thus when the variable $x$ runs over
$\mathbb R^3,$ then the variable $z$ takes all values in $\mathbb R^3,$ hence the inequality $f(x,y)=zAz^T\geq 0$ implies that $A$ is positive semi-definite.
Concluding, we obtain, that $f(\xi)$ is polyconvex then. The proof is finished now.
\end{proof}

\section*{Acknowledgements}
 The authors are grateful to Marin Petkovic for fruitful discussions and the anonymous referee for many valuable comments that have improved the presentation of the manuscript. National Science Foundation for support through grant DMS-1211359 is also appreciated.

\end{document}